\documentclass[reqno]{amsart}

\usepackage{amsmath,amsfonts,amssymb}
\usepackage{graphicx, float}
\usepackage[percent]{overpic}
\usepackage{enumitem}
\usepackage{xcolor}

\title[Graphs of Large Girth and Surfaces of Large Systole]{Graphs of Large Girth\\ and Surfaces of Large Systole}

\author{Bram Petri and Alexander Walker}
\address{\newline\noindent 
Bram Petri\newline
Max Planck Institute for Mathematics, Bonn, Germany \newline
\newline
Alexander Walker\newline 
Mathematics Department, Brown University, Providence, USA}
\thanks{First author partially supported by Swiss National Science Foundation grant number PP00P2\_128557.}

\date{\today}

\newtheorem{thm}{Theorem}[section]
\newtheorem{prp}[thm]{Proposition}
\newtheorem{cor}[thm]{Corollary}
\newtheorem{lem}[thm]{Lemma}

\newtheorem*{clm}{Claim}

\theoremstyle{definition}
\newtheorem{dff}{Definition}[section]

\newcommand{\seq}[3]{\left\{#1\right\}_{#2}^{#3}}

\newcommand{\abs}[1]{\left| #1 \right|}
\newcommand{\card}[1]{\left| #1 \right|}
\newcommand{\tr}[1]{\mathrm{tr}\left(#1\right)}

\newcommand{\sys}{\mathrm{sys}}
\newcommand{\N}{\mathbb{N}}
\newcommand{\Z}{\mathbb{Z}}
\newcommand{\SL}{\mathrm{SL}}
\newcommand{\SemiGroup}[1]{\langle #1\rangle^+}

\begin{document}

%%%%%%%%%%%%%%%%%%%
%   A B S T R A C T
%%%%%%%%%%%%%%%%%%%
\begin{abstract}
The systole of a hyperbolic surface is bounded by a logarithmic function of its genus. This bound is sharp, in that there exist sequences of surfaces with genera tending to infinity that attain logarithmically large systoles. These are constructed by taking congruence covers of arithmetic surfaces. 

In this article we provide a new construction for a sequence of surfaces with systoles that grow logarithmically in their genera. We do this by combining a construction for graphs of large girth and a count of the number of $\SL_2(\Z)$ matrices with positive entries and bounded trace.
\end{abstract}

%%%%%%%%%%%%%%%%%%%
%   T I T L E
%%%%%%%%%%%%%%%%%%%
\maketitle

%%%%%%%%%%%%%%%%%%%
%   I N T R O
%%%%%%%%%%%%%%%%%%%
\section{Introduction}

This article is about a classical problem in hyperbolic geometry and its analogue in graph theory. In the language of graph theory, this problem concerns the existence of regular graphs with large \emph{girth} (a graph is called $k$-regular if all of its vertices have degree $k$). Here, the girth of a graph is the length of its shortest cycle. It follows from an easy counting argument that the girth $h(\Gamma)$ of a $k$-regular graph $\Gamma$ is bounded from above by\footnote{Here and throughout, the notation $f(n)\lesssim g(n)$ indicates that $\limsup_{n \to \infty} f(n)/g(n)\leq 1$.}
\begin{equation*}
h(\Gamma) \lesssim 2\log_{k-1}(n),
\end{equation*}
in which $n$ is the number of vertices of $\Gamma$. Surprisingly, it is actually possible to construct sequences of $k$-regular graphs with girth that grows logarithmically in the number of vertices.  The first constructions of such graphs are due to Erd\H{o}s and Sachs \cite{ES} and Sauer \cite{Sau}, which provide graphs of girth roughly $\log_{k-1}(n)$. The best known constructions have asymptotic girth
\begin{equation*}
\frac{4}{3}\log_{k-1}(n),
\end{equation*}
as $n\rightarrow\infty$. Examples of graphs that achieve this growth are the trivalent sextet graphs of \cite{BiHo}, as proved by Weiss in \cite{Wei}, and the Ramanujan graphs of Lubotzky, Philips, and Sarnak \cite{LPS}. It is not known whether the constant $\frac{4}{3}$ is optimal. For a survey on constructions of graphs of large girth, see \cite{Big}.

From the perspective of hyperbolic geometry, this question turns into a search for genus $g$ (either closed or of finite area) hyperbolic surfaces of large \emph{systole}. Here, the systole of a hyperbolic surface is the length of a shortest homotopically non-trivial and non-peripheral\footnote{Recall that non-peripheral means not homotopic to a puncture or boundary component.} curve.  

Borrowing familiar arguments from the graph case, Buser proved that the systole $\sys(S)$ of a closed hyperbolic surface $S$ satisfies
\begin{equation*}
\sys(S) \lesssim 2\log(g),
\end{equation*}
where $g$ is the genus of $S$. A similar bound holds true when $S$ has punctures, however the proof in this case is less straightforward. The best known upper bounds are due to Schmutz-Schaller \cite{SS} and Fanoni and Parlier \cite{FP}. 

As with graphs, there exist sequences of hyperbolic surfaces with systoles that grow logarithmically in the genus. Curiously, the best known constructions in this case also come with a factor $\frac{4}{3}$. That is, there exist sequences of hyperbolic surfaces $\seq{S_k}{k=0}{\infty}$ such that
\begin{equation*}
\sys(S_k) \gtrsim \frac{4}{3}\log(g_k),
\end{equation*}
where $g_k$ is the genus of $S_k$ and $g_k\rightarrow\infty$ as $k\rightarrow\infty$. Buser and Sarnak \cite{BS} were the first to construct such sequences, using congruence covers of specific closed arithmetic surfaces. Katz, Schaps, and Vishne \cite{KSV} generalized their construction to principle congruence covers of any closed arithmetic surface. In \cite{Mak}, Makisumi proved that the constant $\frac{4}{3}$ is actually optimal for congruence covers. For a survey on surfaces of large systole, see \cite{Par}.

Especially for closed surfaces, very few explicit examples of global and local maximizers of the systole as a function on moduli spaces of closed hyperbolic surfaces are known. The global maximizer is known only in genus $2$ \cite{Jen} and examples of local maximizers are known in genus $3,6$ and $10$ \cite{Ham}. For cusped surfaces, we know an infinite sequence of global maximizers: Schmutz-Schaller proved that the principal congruence subgroups of $\mathrm{PSL}_2(\Z)$ are global maximizers in their moduli spaces \cite{Sch}. 

The main goal of this paper is to give a new construction of sequences of hyperbolic (both closed and cusped) surfaces with systoles that grow logarithmically in their genera (see Corollary \ref{cor_mainresult}). The idea is to combine the graph theoretical construction by Erd\H{o}s and Sachs with a count of the number of matrices of small trace in the semigroup of $\SL_2(\Z)$-matrices with non-negative entries (see Proposition \ref{prop_trace_growth}). 

Concretely, we construct cusped surfaces with systole at least:
\begin{equation*}
\log g-\log\log g-C,
\end{equation*}
where $g$ is the genus of the corresponding surface and $C$ is some absolute constant. Furthermore, given natural numbers (traces) $k_1,\ldots,k_r$, each exceeding the trace corresponding to the systole, and natural numbers (multiplicities) $m_1,\ldots,m_r$ that are small enough (see Section \ref{sec_construction}), we can construct these surfaces in such a way that they have at least $m_i$ curves of length
\begin{equation*}
2\cosh^{-1}(k_i/2)
\end{equation*}
for $i=1,\ldots,r$. These surfaces can be compactified, in essence by adding points in the cusps, in order to obtain closed surfaces. A result of Brooks (Lemma \ref{lem_Brooks}) implies that the systole of these compactified surfaces remains close to their cusped counterparts.

As a consequence of their construction, these surfaces come with a triangulation that has a dual graph $\Gamma$ of girth
\begin{equation*}
h(\Gamma) \gtrsim \tfrac{1}{2}\log_{\phi}(n) \approx 1.039 \log n,
\end{equation*}
where $n$ is the number of vertices of $\Gamma$ and $\phi=(1+\sqrt{5})/2$ denotes the golden ratio.

%%%%%%%%%%%%%%%%%%%
%		A C K N O W L E D G E M E N T
%%%%%%%%%%%%%%%%%%%
\subsection*{Acknowledgement}
Part of this research was carried out while the first author visited the Mathematics Department of Brown University. He thanks the Mathematics Department and in particular his host Jeff Brock for the hospitality during his stay. He is also grateful to the Swiss National Science Foundation for making this stay possible.

We would like to thank Ursula Hamenst\"adt, Ilya Gekhtman, Hugo Parlier, and Peter Sarnak for useful conversations.

%%%%%%%%%%%%%%%%%%%
%   P R E L I M I N A R I E S
%%%%%%%%%%%%%%%%%%%
\section{Preliminaries}

In this section we explain how to construct a surface from a cubic ribbon graph and how the geometry of such a surface depends on the combinatorics of the underlying graph. The construction we use is taken from \cite{BM} (see also \cite{Bro}).
%%%%%%%%%%%%%%%%%%%
%   Surfaces from graphs
%%%%%%%%%%%%%%%%%%%
\subsection{Surfaces from graphs}
We begin with the definition of a cubic ribbon graph:
\begin{dff} A graph $\Gamma=(V,E)$ is called \emph{cubic} if:
\[\deg(v) = 3 \text{ for all }v\in V.\]
A \emph{cubic ribbon graph} is a pair $(\Gamma,\mathcal{O})$, where $\Gamma$ is a cubic graph and $\mathcal{O}$ is a map that assigns a cyclic order to the triple of edges emanating from each vertex.
\end{dff}

Ribbon graphs are sometimes called \emph{fatgraphs} or \emph{oriented graphs}. It should be noted that these definitions do not distinguish between graphs and multigraphs: graphs in this text are allowed to have loops and multiple edges.

Given a cubic ribbon graph $(\Gamma,\mathcal{O})$ we construct a topological surface $S(\Gamma,\mathcal{O})$ as follows. To every vertex $v$ of $\Gamma$ we assign a triangle, whose three sides correspond to the edges emanating from $v$. Then, for each edge $e$ of $\Gamma$, we glue together the two triangle sides corresponding to $e$. We do this in such a way that the resulting surface is orientable. The orientation we pick is the one corresponding to the cyclic order on the edges emanating from each vertex (eg. via the right-hand rule). In this way, the pair $(\Gamma,\mathcal{O})$ uniquely determines the surface $S(\Gamma,\mathcal{O})$.

We observe that the graph $\Gamma$ naturally embeds into the surface $S(\Gamma,\mathcal{O})$, and we will often think of it as such without mention.  In this setting the combinatorics of the underlying graph can be used to control the topology of the surface; for example, the following result of Beineke and Harary bounds the genus of $S(\Gamma,\mathcal{O})$ using solely combinatorial data:

\begin{prp}\cite{BeHa}\label{prp_graphgen} Let $\Gamma$ be a connected graph with $p$ vertices, $q$ edges and girth $h$ that embeds into a surface $S$ of genus $g$. Then:
$$
1+\frac{1}{2}\left(1-\frac{2}{h}\right)q-\frac{1}{2}p \; \leq \; g
$$
\end{prp}

We can turn $S(\Gamma,\mathcal{O})$ into a geometric surface by defining a metric on the underlying triangle. Here, the metric we choose is the metric of the ideal hyperbolic triangle. This metric extends to all of $S(\Gamma, \mathcal{O})$ so long as the gluing maps between the triangles are isometries. There are an uncountable number of such isometries, and in this article we choose the unique gluing such that the perpendiculars connecting the identified sides to the vertices opposite them meet, as in Figure \ref{pic_gluing} below:

\begin{figure}[H]
\begin{center}
\begin{overpic}[scale=1]{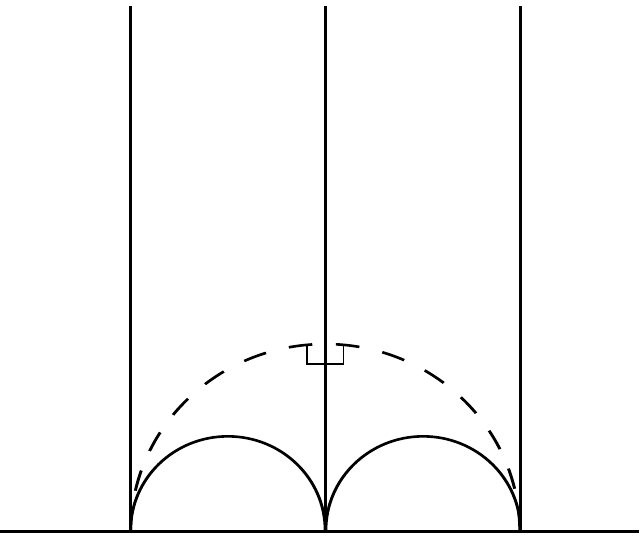}
\put (5,70) {$\mathbb{H}^2$}
\end{overpic}
\end{center}
\caption{A shear $0$ gluing of two triangles.}
\label{pic_gluing}
\end{figure}

In a general gluing of ideal hyperbolic triangles, the signed distance between the two endpoints of the perpendiculars is called the \emph{shear} of the gluing. In other words, we choose each of our gluings to have shear $0$. This implies that $S(\Gamma,\mathcal{O})$ is a punctured surface with a complete hyperbolic structure. The punctures of $S(\Gamma, \mathcal{O})$, called cusps, correspond one-to-one with cycles in $(\Gamma,\mathcal{O})$ that consist exclusively of lefthand turns.

It is also possible to turn $S(\Gamma,\mathcal{O})$ into a closed hyperbolic surface. To see this, we recognize the hyperbolic structure on $S(\Gamma,\mathcal{O})$ as a conformal structure in which the cusps are conformally equivalent to punctured disks. By adding points to these punctured disks, we can extend this conformal structure to a conformal structure of a closed surface. The uniformization theorem gives the existence of a unique complete hyperbolic structure in the equivalence class of this conformal surface (provided that the genus of $S$ is at least $2$, which we assume). Our surface, equipped with this specific hyperbolic structure, will be denoted $\overline{S}(\Gamma,\mathcal{O})$.

%%%%%%%%%%%%%%%%%%%
%   The geometry of curves
%%%%%%%%%%%%%%%%%%%
\subsection{The geometry of curves}

A particularly nice feature of the construction above is that the geometry of $S(\Gamma,\mathcal{O})$ is entirely determined by the combinatorics of the underlying ribbon graph $(\Gamma,\mathcal{O})$. In what follows we will explain how to understand the geometry of curves on $S(\Gamma,\mathcal{O})$ in terms of $(\Gamma,\mathcal{O})$.

It is a classical fact from hyperbolic geometry (see for example Theorem 1.6.6 in \cite{Bus}) that the homotopy class of a non-peripheral, non-trivial closed curve $\gamma$ contains a unique geodesic $\widetilde{\gamma}$. This geodesic minimizes the length among all curves in that homotopy class. Furthermore, if the surface is given by
\begin{equation*}
S=\mathbb{H}^2/G
\end{equation*}
where $G \subset \mathrm{PSL}_2(\mathbb{R})$ is a torsion-free discrete subgroup, then we can find a hyperbolic element $g\in G$ such that the axis of $g$ in $\mathbb{H}^2$ projects to the $\widetilde{\gamma}$ in $S$. The length of $\tilde{\gamma}$ is equal to the translation length $T_g$ of $g$. This translation length is in turn given by:
\begin{equation*}
T_g = 2\cosh^{-1} \left( \frac{\abs{\tr{g}}}{2}\right),
\end{equation*}
in which $\tr{g}$ denotes the trace of $g \in \mathrm{PSL}_2(\mathbb{R})$.  

The construction of $S(\Gamma,\mathcal{O})$ allows us to recover the translation length for a given curve. Namely, given a curve $\gamma$ on $S(\Gamma,\mathcal{O})$, we can homotope it to a cycle\footnote{Here, a \emph{cycle} is any closed path on a graph. A closed path that meets every vertex and edge in $\Gamma$ at most once will be called a \emph{circuit}.} $\tilde{\gamma}$ on the graph $\Gamma \subset S(\Gamma,\mathcal{O})$. Because $\Gamma$ is oriented we can tell whether $\gamma$ turns `left' or `right' at a given vertex. This means that traversing the curve $\gamma$ once gives us a word $w(\gamma)$ in letters $L$ and $R$. We set
\begin{equation} \label{lr_def}
L=\left(\begin{matrix} 1 & 1 \\ 0 & 1 \end{matrix}\right) \text{ and } R=\left(\begin{matrix}1 & 0 \\ 1 & 1 \end{matrix}\right),
\end{equation}
which turns $w(\gamma)$ into a matrix. The length of the unique geodesic $\tilde{\gamma}$ homotopic to $\gamma$ is now given by
\begin{equation*}
\ell(\tilde{\gamma}) := 2\cosh^{-1} \left( \frac{\tr{w(\gamma)}}{2}\right).
\end{equation*}
Note that the word $w(\gamma)$ is not well-defined: it depends on where we start to traverse $\gamma$ and in which direction we do so. This has no effect on the trace, however, which means that $\ell(\tilde{\gamma})$ \emph{is} well-defined.

When we consider $\overline{S}(\Gamma,\mathcal{O})$ instead, we seem to lose the combinatorial description of the length of curves: the uniformization theorem tells us that there is a natural hyperbolic structure on $\overline{S}(\Gamma,\mathcal{O})$, but it does not tell us anything about the geometry of this structure. However, there is a way to solve this problem, using a result of Brooks \cite{Bro}. We first need the following definition:
\begin{dff} Let $S$ be a hyperbolic surface with at least one cusp and fix $r >0$. Then $S$ is said to have \emph{cusp length} $\geq r$ if for each cusp of $S$ there exists a non-self-intersecting horocycle of length $r$ about that cusp, such that no two horocycles intersect.\end{dff}

We have the following Lemma:
\begin{lem}\label{lem_Brooks} \cite{Bro} Let $(\Gamma,\mathcal{O})$ be a cubic ribbon graph such that $S(\Gamma,\mathcal{O})$ has cusp length $\geq r$, where $r$ is sufficiently large. Then for every non-peripheral and homotopically essential geodesic $\overline{\gamma}$ on $\overline{S}(\Gamma,\mathcal{O})$ there exists a geodesic $\gamma$ on $S(\Gamma,\mathcal{O})$ such that the image of $\gamma$ under the map  $S(\Gamma,\mathcal{O})\rightarrow  \overline{S}(\Gamma,\mathcal{O})$  is homotopic to $\overline{\gamma}$ and
$$
\ell(\overline{\gamma}) \leq \ell(\gamma) \leq (1+\delta(r)) \ell(\overline{\gamma}),
$$
where $\delta(r)\rightarrow 0$ as $r\rightarrow\infty$.
\end{lem}

%%%%%%%%%%%%%%%%%%%
%  W O R D S
%%%%%%%%%%%%%%%%%%%
\section{Words in $L$ and $R$}

In the previous section we have seen that in order to understand the lengths of curves on a surface $S(\Gamma,\mathcal{O})$ we need to understand which words in $L$ and $R$ appear as cycles on $(\Gamma,\mathcal{O})$. In this section we collect some basic properties of the set of words in $L$ and $R$ and use these properties to produce estimates for the number of words in $L$ and $R$ of bounded trace.

%%%%%%%%%%%%%%%%%%%
%  Basic properties
%%%%%%%%%%%%%%%%%%%
\subsection{Basic properties}

We start with some notation. The semigroup of words in $L$ and $R$ (as defined in line \eqref{lr_def}) will be denoted
\begin{equation*}
\SemiGroup{L,R}.
\end{equation*}
Elements of this set will sometimes be interpreted as matrices and sometimes as strings in two letters. It will be clear from the context which of the two we mean. We define an equivalence relation on $\SemiGroup{L,R}$ as follows:
\begin{dff} Let $w,w'\in\SemiGroup{L,R}$. We write $w\sim w'$ if either of the following is true:
\begin{itemize}[leftmargin=0.2in]
\item[--] $w'$ is a cyclic permutation of $w$.
\item[--] $w'$ is a cyclic permutation of $w^*$, where $w^*$ is the word obtained by reading $w$ backwards and interchanging $L$ and $R$.
\end{itemize}
\end{dff}

In the previous section we noted that the map from cycles on $(\Gamma,\mathcal{O})$ to $\SemiGroup{L,R}$ is not well-defined. The equivalence defined here solves this problem, since the compositions of maps into $\SemiGroup{L,R}/\sim$ \emph{is} well-defined. If $\gamma$ is a cycle on $\Gamma$ and its image under the given map is $[w]$ we say that $\gamma$ \emph{carries} $[w]$.

The following lemma provides various estimates on the trace of words in $\SemiGroup{L,R}$:

\begin{lem}\label{lem_basic} Let $w\in \SemiGroup{L,R}$. If $w'\in \SemiGroup{L,R}$ can be obtained from inserting letters from $\{L,R\}$ into $w$ then:
\[\tr{w'} \geq \tr{w}.\]
Moreover, $\tr{w} \leq \phi^{\mathrm{len} w}+1$ in general, and $\tr{w} \geq \mathrm{len}(w)+1$ unless $w \sim L^m$ for some $m$.
\end{lem}
\begin{proof} By the identity
\begin{align} \label{left_mult_identity}
\mathrm{tr}\left(\left(\begin{matrix} 1 & 1 \\ 0 & 1 \end{matrix}\right)\left(\begin{matrix} a & b \\ c & d \end{matrix} \right)\right) = \mathrm{tr}\left(\begin{matrix} a & b \\ c & d \end{matrix}\right) +c,
\end{align}
we see that $\tr{Lw} \geq \tr{w}$ with equality if and only if $c=0$, ie. if and only if $w = L^m$ for some $m$. Since any letter insertion is equivalent (modulo $\sim$) to left-multiplication by $L$ and trace is well-defined on $\SemiGroup{L,R}/\sim$, we see that trace is non-decreasing with letter insertion.

To prove the upper bound on $\tr{w}$ we use the fact that among all words of $k$ letters, the trace is maximized by words of the form $(LR)^m$ and $R(LR)^m$ (depending on whether $k$ is even or odd). This in turn follows from an elementary but tedious case by case analysis. The fact that the traces of  $(LR)^m$ and $R(LR)^m$ satisfy the inequality follows by direct computation.

If $w \not\sim L^m$, choose $\eta \sim w$ ending in $LR$. For each letter of $\eta$ prepended to $LR$, we augment our trace by at least $1$. (In the case of multiplication by $L$, this is line \eqref{left_mult_identity}.) It follows that $\tr{\eta} \geq \mathrm{len}(\eta)-2+ \tr{LR} = \mathrm{len}(\eta)+1$.
\end{proof}

In particular, the insertion of letters into a word cannot decrease its trace. From this, we for instance conclude that the systole of $\overline{S}(\Gamma,\mathcal{O})$ is always homotopic to a circuit in $\Gamma$.

%%%%%%%%%%%%%%%%%%%
%  Counting
%%%%%%%%%%%%%%%%%%%
\subsection{Counting words by trace}

Let $\SL_2(\Z)^+$ denote the semigroup of integer matrices with determinant $1$ and non-negative coordinates. It's clear that $\SL_2(\Z)^+$ contains $\SemiGroup{L,R}$.

Let $n(m)$ denote the number of elements in $\SemiGroup{L,R}$ of trace $m$. The infinite collections $\SemiGroup{L}$ and $\SemiGroup{R}$ demonstrate that $n(2)$ is not finite. On the other hand, $n(m)$ is finite for $m \geq 3$, as the following Proposition shows:

\begin{prp} \label{upper_bound_prop}
The following inequality holds for all $m \geq 3$:
\begin{align*}
n(m) \leq \sum_{a=1}^{m-1} d\left(a(m-a)-1\right),
\end{align*}
in which $d(k)$ is the number of divisors of $k$.
\end{prp}
\begin{proof} Let $n'(m)$ denote the number of elements of $\SL_2(\Z)^+$ of trace $m$. Then $n(m) \leq n'(m)$, whenever both are finite. To enumerate the elements of $\SL_2(\Z)^+$ of trace $m$, consider a general matrix
\begin{align} \label{matrix form}
\gamma :=\left(\begin{matrix} a & b \\ c & d \end{matrix}\right) \in \SL_2(\Z)^+
\end{align}
of trace $m$. Free choice of $a$ in the interval $[1,m-1]$ determines $d=m-a$ by trace. As well, the determinant relation $ad-bc=1$ gives $ad-1=bc$, hence $b$ divides $a(m-a)-1$ (whereafter choice of $b$ determines $c$ uniquely). It follows that
\[n(m) \leq n'(m) = \sum_{a=1}^{m-1} d(a(m-a)-1),\]
as desired. 
\end{proof}

As it happens, the upper bound in  Proposition \ref{upper_bound_prop} is an equality. This follows from the non-obvious fact that the inclusion $\SemiGroup{L,R} \subset \SL_2(\Z)^+$ is an equality:

\begin{prp} We have $\SemiGroup{L,R} = \SL_2(\Z)^+$.
\end{prp}

\begin{proof} Fix $\gamma \in \SL_2(\Z)^+$, defined as in line \eqref{matrix form}. For $\gamma$ of trace $2$, we see that $\gamma$ takes the form $L^k$ or $R^k$ for some integer $k \geq 0$, hence $\gamma \in \SemiGroup{L,R}$. Now, suppose by induction that $\gamma \in \SemiGroup{L,R}$ for all matrices of trace less than $m>2$ and fix $\gamma$ of trace $m$.

Computation shows that $L^{-1} \gamma \in \SL_2(\Z)^+$ provided $a > c$ and $b \geq d$, while $R^{-1} \gamma \in \SL_2(\Z)^+$ under the assumptions $c \geq a$ and $d>b$. For the sake of contradiction, suppose that neither holds. If $a > c$ and $d>b$, then
\[\det(\gamma) = ad-bc \geq (c+1)(b+1)-bc=b+c+1=1,\]
so that $b=c=0$ and $\gamma = I$, which contradicts that $\mathrm{tr}(\gamma) =m$. Alternatively, suppose that $c \geq a$ and $b \geq d$. Then
\[\det(\gamma) = ad-bc \leq cb-bc = 0,\]
another contradiction. Thus $\SL_2(\Z)^+$ contains at least one of $L^{-1} \gamma$ or $R^{-1} \gamma$, hence so does $\SemiGroup{L,R}$ (by induction). In either case, it follows that $\gamma \in \SemiGroup{L,R}$.
\end{proof}

\vspace{3 mm}
To estimate the growth of $n(m)$ we need only estimate the growth of the divisor function. It is well-known that $d(n) =O(n^\varepsilon)$ for all $\varepsilon > 0$, from which we immediately obtain that $n(m) = O(m^{1+\varepsilon})$. 

We denote by $N(m)$ the partial sums of $n(m)$, ie.
\[N(m)=\sum_{k=3}^m n(k).\]
This function satisfies the trivial upper bound $N(m) = O(m^{2+\epsilon})$. With a bit more work, we obtain a more precise upper bound for this function:

\begin{prp} The function $N(m)$ is $O(m^2 \log m)$. \label{prop_trace_growth}
\end{prp}

\begin{proof} Interchanging the order of summation, we write
\begin{align} \label{Nm_growth_1}
N(m) = \sum_{k \leq m} \sum_{a=1}^{k-1} d(a(k-a)-1) = \sum_{a=1}^{m-1} \sum_{n\equiv -1 (a)}^{a(m-a-1)} d(n),
\end{align}
so that the inner sum adds the contribution of the divisor function over the arithmetic progression $n \equiv -1 \mod a$. To continue, we require a well-known result due independently to Selberg and Hooley; that the Weil bound for the Kloosterman sum gives a uniform estimate
\[\sum_{\substack{n \leq X \\n\equiv -1(a)}} d(n) = \frac{1}{\varphi(a)} \sum_{\substack{n\leq X \\ (n,a)=1}} d(n) + O\left((a^{\frac{1}{2}}+X^{\frac{1}{3}})X^{\epsilon}\right),\]
as recounted by Fouvry and Iwaniec in \cite{FI}, eg. Opening up the second divisor sum in the line above yields
\begin{align} \label{Nm_growth_2}
\frac{1}{\varphi(a)} \sum_{\substack{m_1\leq X \\ (m_1,a)=1}} \sum_{\substack{ m_2 \leq X/m_1 \\(m_2,a)=1}} \!\!1 =  \sum_{\substack{m_1\leq X \\ (m_1,a)=1}} \frac{X}{a m_1} +O\left(1\right) \ll \frac{X \log X}{a},
\end{align}
in which we've abandoned coprimality to simplify our bound. Returning to line \eqref{Nm_growth_1}, we take $X:=a(m-a-1)-1$, and end with the estimate
\[N(m) \ll \sum_{a < m} \frac{a(m-a)}{a}\log(a(m-a)) = O\left(m^2 \log m\right),\]
as desired.
\end{proof}

\textit{Remark ---} Conversely, restricting the final $m_1$-sum in line \eqref{Nm_growth_2} to the primes less than $X$ gives a lower bound in line in \eqref{Nm_growth_2} of the form $X (\log\log X)/a$. This propagates to show that $N(m) \gg m^2 \log\log m$, ie. that $N(m)$ grows super-quadratically.

%%%%%%%%%%%%%%%%%%%
%   C O N S T R U C T I O N
%%%%%%%%%%%%%%%%%%%
\section{Construction of surfaces with large systole} \label{sec_construction}

The results from the previous section allow us to construct hyperbolic surfaces with systole logarithmic in their genus. The construction we present is a Riemann surface version of the construction for graphs of large girth by Erd\H{o}s and Sachs \cite{ES} and Sauer \cite{Sau}, while our presentation is inspired by a version of these proofs given by Bollob\'as \cite{Bol} (see also \cite{Big}). 

In our theorem below we will speak of oriented circuits without reference to the specific oriented trivalent graphs containing them. An oriented circuit in this sense will be a circuit in which it is known whether one turns right or left when traversing a vertex in a given direction. Note that such a circuit naturally corresponds to a word in $L$ and $R$. In this way we are able to define the trace of a cycle or an oriented path. The data of an oriented path in this context includes the data of the direction the path turns in at its initial and final vertices. This means that an oriented path naturally runs between two edges instead of two vertices.

We have the following theorem:

\begin{thm}\label{thm_construction} Let $\varepsilon>0$ and let $H=(V,E)$ be an oriented graph in which every connected component is an oriented circuit such that:
\begin{itemize}[leftmargin=0.3in]
\item[1.] Each circuit in $H$ has trace at least $k$.
\item[2.] $H$ has an even number of vertices at least equal to $2\cdot N(k-2)+4k-4$.
\end{itemize}
Then we can complete $H$ to a trivalent ribbon graph $H'$, respecting the orientation of the circuits, in such a way that $H'$ contains no homotopically essential non-peripheral cycle carrying a word of trace less than $k$. Furthermore, the girth of $H'$ satisfies
$$
h(H') \gtrsim \log_{\phi}(k),
$$
and every cusp in $H'$ has at least $k$ triangles around it.
\end{thm}

\begin{proof} Our proof is constructive. We will construct a set $E'$ of edges on $V$ that contains $E$ such that the graph $H'=(V,E')$ has the desired properties. Note that this construction does not add vertices, hence the orientation at every vertex in $H'$ is given by the initial data and does not change anywhere in the process that we will describe. As such, we will not mention it in the rest of the proof.

We shall consider sets $E'$ of edges on $V$ with the following properties:
\begin{itemize}
\item[(a)] $E\subset E'$.
\item[(b)] Every cycle in $H'=(V,E')$ has trace at least $k$ or is a cycle of lefthand turns with at least $k$ edges.
\end{itemize} 
An example of such a set is the set $E$ itself.

We will prove the following claim, which is sufficient to prove the theorem:
\begin{clm}
If $E'$ satisfies (a) and (b) and the graph $H'=(V,E')$ has a vertex of degree $2$ then there exists a set $E''$ of edges on $V$ that also satisfies (a) and (b) such that $\card{E''}=\card{E'}+1$.
\end{clm}

In order to prove the claim we need to define some specific subsets of $V$, depending on $E'$. Given $x,y\in V$, a \emph{forbidden $k$-path} between a degree $2$ vertex $x$ and a vertex $y$ will be a path in $H'$ of trace less than $k-1$ and with fewer than $k-1$ edges. We emphasize again that the trace of a path depends on the choice of the directions of the turns at $x$ and $y$. At $x$ we we will always choose the direction so that the path comes from the `missing' edge at $x$. This makes sense because of the fact that the orientation at $x$ is predefined. For a degree $2$ vertex $x$ we now define:
\begin{equation*}
F_k(x) = \{y\in V : \text{there exists a forbidden }k\text{-path from }x\text{ to } y \}.
\end{equation*}
This set can be seen as some sort of $k$-ball around the missing edge at $x$.

The crucial observation is the following: if $x$ and $y$ have degree $2$ and $y \in V \smallsetminus F_k(x)$, then we do not introduce any (non-lefthand turn) cycles of trace $\leq k$ by adding an edge between $x$ and $y$.

To see this, note that such a cycle necessarily builds upon a non-forbidden $k$-path from $x$ to $y$, hence has length $\geq k$ or trace $ \geq k$. This last case uses Lemma \ref{lem_basic}, and gives our claim directly. If our cycle has length $\geq k$, then it carries a word $w$ such that either
\pagebreak
\begin{enumerate}[leftmargin=0.3in]
\item[1.] $w \sim L^m$ with $m \geq k$, in which case the cycle corresponds to a cusp.
\item[2.] $w \not\sim L^m$, whereby Lemma \ref{lem_basic} gives the inequality $\tr w \geq \mathrm{len}(w)+1$, which implies that our cycle has trace $\geq k$.
\end{enumerate}

We note as well that since the total number of vertices is even there will always be an even number of degree $2$ vertices. So, given $E'$, there are two cases:
\begin{itemize}[leftmargin=0.6in]
\item[\textbf{Case 1.}] There exist two degree $2$ vertices $x$ and $y$ in $H'$ for which there does not exist a forbidden $k$-path connecting $x$ to $y$.
\item[\textbf{Case 2.}] There exists no pair of such vertices.
\end{itemize}
Our proof breaks into cases along these lines.

Case $1$ is immediate. When we connect $x$ and $y$ by an edge, we obtain a set $E''$ that satisfies our requirements.

For Case $2$ our argument is more involved. If $y \in F_k(x)$, then $x$ and $y$ may be joined by a path of trace at most $k-2$ and with at most $k-2$ edges. There are $N(k-2)$ paths of trace in $[3,k-2]$, and there are at most $2(k-2)+1$ paths of trace $2$ and length $\leq k-2$. (Coming from the trivial path and the $2(k-2)$ paths of the form $L^m$ and $R^m$ with $m \in [1,k-2]$.) Thus
\begin{equation*}
\vert F_k(x) \vert \leq N(k-2) +2k-3.
\end{equation*}

Now take $x,y\in V$ two vertices of degree $2$ in $H'$ and define the sets
\begin{equation*}
U = F_k(x)\cup F_k(y) \quad \text{and}\quad I=F_k(x)\cap F_k(y).
\end{equation*}
Inclusion-exclusion gives
\begin{align*}
\card{U} & = \card{F_k(x)} + \card{F_k(y)} - \card{I} \\
& \leq 2N(k-2)+4k-6-\card{I} \\
& \leq \card{V}-2-\card{I}.
\end{align*}
Then, defining $W := V \setminus U$, we have
\begin{equation*}
\card{W} = \card{V} - \card{U} \geq \card{I} + 2.
\end{equation*}
Under the assumptions of Case $2$, all vertices in $W$ have degree $3$ in $H'$. In particular, each vertex in $W$ is an endpoint to a unique edge in $E'\setminus E$. Thus for every $w\in W$ there exists a unique vertex $w'$ such that $w$ and $w'$ are endpoints of an edge in $E'\backslash E$. Note as well that $w \neq w'$, as equality forces $\deg w \geq 4$, a contradiction. Using this we define the set $W'$:
\begin{equation*}
W' := \{w'\in V : \exists\, w\in W\text{ such that }w'\text{ and }w\text{ share an edge in }E'\backslash E \}.
\end{equation*}
We have
\begin{equation*}
\card{W'}=\card{W} \geq \card{I}+2.
\end{equation*}
Thus there exists some $w'\in W'$ not in $I$. In other words, there is either no forbidden $k$-path from $x$ to $w'$ or no forbidden $k$-path from $y$ to $w'$. Without loss of generality we assume the former. We now define the edge set
\begin{equation*}
E'' := E'\setminus ww' \cup xw' \cup yw,
\end{equation*}
and claim that $E''$ satisfies (a) and (b). 

Condition (a) is immediate. For (b), we proceed by contradiction. Suppose $H''=(V,E'')$ contains a non-trivial, non-peripheral cycle of trace $< k$. This cycle necessarily contains both of the edges $xw'$ and $yw$. There are two options for the order of appearance of the vertices $x$, $y$, $w$, and $w'$ along this cycle:
\begin{align} \label{two_options}
x,w',y,w \;\; \text{ or } \;\; w',x,y,w
\end{align}
(up to the dihedral symmetry of the cycle).

Since we have assumed that there is no forbidden $k$-path between $x$ and $w$ in $H'$, any non-trivial, non-peripheral cycle containing $x$ and $w$ in succession has trace at least $k$. This implies that the first option is impossible (by Lemma \ref{lem_basic}).

For the second option, consider the following diagram:
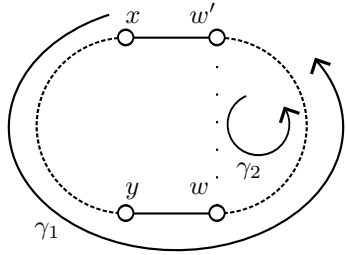
\begin{figure}[H]
\begin{center}
\begin{overpic}[scale=1]{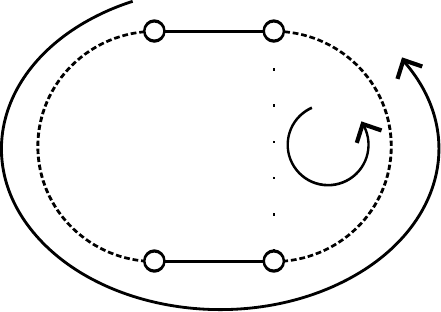}
\put (8,5) {$\gamma_1$}
\put (68,23) {$\gamma_2$}
\put (35,17) {$y$}
\put (54,17) {$w$}
\put (54,68) {$w'$}
\put (35,68) {$x$}
\end{overpic}
\end{center}
\caption{The second option in line \eqref{two_options}.}
\label{pic_twocircs}
\end{figure}
In Figure \ref{pic_twocircs} above, $\gamma_1$ represents the offending cycle of trace $< k$ in $H''$, while the cycle $\gamma_2$ lies in $H'$. Because the turns
\begin{equation*}
w'\rightarrow x \;\;\text{and}\;\; w\rightarrow y
\end{equation*}
are in the same direction as the turns
\begin{equation*}
w'\rightarrow w \;\;\text{and}\;\; w\rightarrow w',
\end{equation*}
respectively, the word on $\gamma_1$ can be obtained by concatenating the word on $\gamma_2$ with some number of letters. 

Since $\gamma_2 \subset H'$, we have two options for the word $w_2$ carried by $\gamma_2$: either $w_2 \sim L^m$ for some $m \geq k$, or $\tr{w_2} \geq k$. Either way, it follows from Lemma \ref{lem_basic} that the word $w_1$ on $\gamma_1$ has $\tr{w_1}>k$ or that $w_1 \sim L^m$ for some $m \geq k$. This contradicts our assumptions on $\gamma_1$, which proves that $E''$ satisfies (b). The chief Claim follows.

Finally, to see that the graph we obtain in the end has girth $\gtrsim \log_{\phi}(k)$  we note that all the circuits in $H'$ are either left hand turn circuits of at least $k$ edges or carry a word of trace at least $k$.  Lemma \ref{lem_basic} tells us that a word with trace at least $k$ has at least  $\sim \log_{\phi}(k)$ letters, which implies the statement.
\end{proof}

As a corollary, we obtain the following:
\begin{cor}\label{cor_mainresult}The construction above gives rise to sequences of cusped hyperbolic surface $\seq{S_{g_k}}{k=0}{\infty}$ and sequences of closed hyperbolic surfaces $\seq{\overline{S}_{g_k}}{k=0}{\infty}$ such that:
\begin{align} \label{systole_asymptotics}
\liminf_{k\rightarrow\infty} \frac{\sys(S_{g_k})}{\log(g_k)}  \geq 1 \quad \text{ and } \quad \liminf_{k\rightarrow\infty} \frac{\sys(\overline{S}_{g_k})}{\log(g_k)}  \geq 1
\end{align}
in which $c_1 k^2 \log\log k\leq g_k \leq c_2 k^2 \log k$ for some absolute constants $c_1,c_2 >0$.
\end{cor}

\begin{proof} Using Theorem \ref{thm_construction} above, we can construct a sequence $\seq{(\Gamma_k,\mathcal{O}_k)}{k=1}{\infty}$ of cubic ribbon graphs such that $\Gamma_k$ has $O\left(k^2 \log k\right)$ vertices and that the trace of any essential non-trivial curve on $S(\Gamma_k,\mathcal{O}_k)$ is bounded from below by $k$. If $V_k$ denotes the number of vertices of $\Gamma_k$, this means that
\begin{align*}
\sys\left(S(\Gamma_k,\mathcal{O}_k)\right) & = 2\cosh^{-1}\left(\min\left\{\tfrac{\tr{\gamma}}{2} : \gamma \text{ a cycle on }\Gamma_k\right\} \right) \\
   & \geq 2\cosh^{-1}\left(\frac{k}{2}\right) \geq \log(V_k)-\log \log k- \log B,
\end{align*}
where $B$ comes from the implied constant in the estimate $V_k = O(k^2 \log k)$. It remains to relate the number of vertices of $(\Gamma_k,\mathcal{O}_k)$ to the genus of $S(\Gamma_k,\mathcal{O}_k)$. By construction, our surface $S(\Gamma_k,\mathcal{O}_k)$ comes with a triangulation, which means that we can compute its Euler characteristic. For each $k$, we have
\begin{equation*}
2c - 2\sum\limits_{i=1}^c g_i = v-e+f
\end{equation*}
where $c$ is the number of connected components of $S(\Gamma_k,\mathcal{O}_k)$, $g_i$ the genus of the $i^{th}$ connected component, $v$ the number of vertices in the triangulation, $e$ the number of edges, and $f$ the number of faces. We have
\begin{equation*}
f = V_k \quad \text{and} \quad e = \frac{3}{2} \cdot V_k,
\end{equation*}
and thus
\begin{equation*}
2\sum_{i=1}^c g_i = 2c - v + \frac{1}{2}V_k.
\end{equation*}
Since $h(\Gamma_k) \gg \log_\phi k$, each connected component of $\Gamma_k$ consists of at least $\log_\phi k$ vertices, hence $V_k \gg c \log_\phi k$. It follows that
\[g_i \leq 2\sum_{i=1}^c g_i \ll \frac{V_k}{\log_\phi k} + \frac{1}{2} V_k = O(V_k),\]
and hence
\begin{align} \label{systole_bound_explicit}
\sys\left(S(\Gamma_k,\mathcal{O}_k)\right) \geq \log(g_k) -\log\log k- R,
\end{align}
in which $R$ is a constant independent of $k$ and $g_k$ is the genus of any of the connected components of $S(\Gamma_k,\mathcal{O}_k)$ or the sum of all these genera, depending on the reader's preference. Line \eqref{systole_bound_explicit} proves the first inequality of line \eqref{systole_asymptotics}; for the second inequality, we need only recall Lemma \ref{lem_Brooks}.

To prove that the genera of these surfaces actually grow super-quadratically in $k$, we use that the fact that $\Gamma_k$ embeds into $S(\Gamma_k,\mathcal{O}_k)$. This puts restrictions on the genus $g_k$ of $S(\Gamma_k,\mathcal{O}_k)$. Concretely, Proposition \ref{prp_graphgen} tells us that
\begin{equation*}
1+\left(\frac{1}{4}-\frac{3}{2h(\Gamma_k)}\right)V_k \leq g_k.
\end{equation*}
Since $h(\Gamma_k) \gg \log k$, the lefthand side above tends to $1+V_k/4$ as $k$ grows large, hence $g_k \gg V_k$, which completes the proof.
\end{proof}

Finally, we remark that Theorem \ref{thm_construction} gives us some control over the bottom part of the length spectrum. That is, given $k_0\leq k_1\leq \ldots \leq k_r$ and $m_1,\ldots,m_r \in\N$ such that there exist words $w_1,\ldots,w_r\in \SemiGroup{L,R}$ satisfying
\begin{equation*}
\tr{w_i}=k_i \;\;\text{and}\;\;\sum_{i=1}^r m_i\,\mathrm{len} (w_i) \leq 2\cdot N(k_0-2)+4k_0-4
\end{equation*}
for $i=1,\ldots,r$, we can construct our cusped surface in such a way that it has systole $\geq 2\cosh^{-1}(k_0/2)$, genus $\ll k_0^2\log k_0$, and contains at least $m_i$ curves of length
\begin{equation*}
2\cosh^{-1}\left(\frac{k_i}{2}\right),
\end{equation*}
for $i=1,\ldots,r$. In the closed case we do not get such exact control, but we can choose to construct the surfaces such that they contain $m_i$ curves with lengths in a small interval around the values above. By Lemma \ref{lem_Brooks} these intervals become arbitrarily small as $k_0$ becomes large.

Given words $w_1,\ldots,w_r$, the condition on the multiplicities $m_i$ is easy to verify. However, without explicit examples of words, it is not easy to see whether a set of traces $k_0,k_1,\ldots,k_r$ and multiplicities $m_1,\ldots,m_r$ is `realizable'. A sufficient (but certainly not necessary) condition on the traces and multiplicities is
\begin{equation*}
\sum_{i=1}^r m_i (k_i-1) \leq 2\cdot N(k_0-2)+4k_0-4
\end{equation*}
This follows from the fact that $\tr{L^{k-2}R}=k$. 

%%%%%%%%%%%%%%%%%%%%%%%%%%%%%%%%%%%%%%%%%%%%%%%%%
%		R E F E R E N C E S
%%%%%%%%%%%%%%%%%%%%%%%%%%%%%%%%%%%%%%%%%%%%%%%%%

\end{document}